\documentclass{amsart}
\usepackage{fdsymbol} 
\usepackage{tikz-cd} 

\usepackage{enumitem}

    \usepackage{etoolbox}
    \patchcmd{\section}{\scshape}{\large\bfseries}{}{}
    \makeatletter
    \renewcommand{\@secnumfont}{\bfseries}
    \makeatother
    
\usepackage{biblatex}
\usepackage{booktabs}
\bibliography{references}

\numberwithin{equation}{section}

\newtheorem{theorem}{Theorem}
\newtheorem*{theorem*}{Theorem}

\newtheorem{lemma}[theorem]{Lemma}

\newtheorem{conjecture}{Conjecture}
\theoremstyle{definition}

\title{On the cokernel of the  Baumslag rationalization}

\author{Sergei O. Ivanov} 
\address{
Laboratory of Modern Algebra and Applications,  St. Petersburg State University, 14th Line, 29b,
Saint Petersburg, 199178 Russia}
\email{ivanov.s.o.1986@gmail.com}

\thanks{The work is supported by: (1) Ministry of Science and Higher Education of the Russian Federation, agreement  075-15-2019-1619; (2) the grant of the Government of the Russian Federation for the state
support of scientific research carried out under the supervision of leading scientists,
agreement 14.W03.31.0030 dated 15.02.2018; (3)  RFBR according to the research project 20-01-00030; (4) Russian Federation President Grant for Support of Young Scientists MK-681.2020.1.
}

\def\bau{{\sf Bau}}
\def\ZZ{\mathbb{Z}}
\def\QQ{\mathbb{Q}}
\def\CC{\mathcal{C}}

\begin{document}
\maketitle
\begin{abstract}
We prove that for the free group of rank two $F$  the cokernel of the homomorphism to its Baumslag rationalization 
$F\to {\sf Bau}(F)$ is not abelian. Moreover, this cokernel contains a free  subgroup of countable rank. This answers a question of Emmanuel Farjoun.  
\end{abstract}

\section*{Introduction}

A localization on a category $\CC$ is a couple $L=(L,\eta),$ where $L:\CC\to \CC$ is a functor and $\eta:{\sf Id}\to L$ is a natural transformation such that the natural transformations $\eta L, L\eta:L\to L^2$ are equal and they are isomorphisms. Equivalently one can define a localization as a monad whose multiplication map is an isomorphism, or as a reflection with respect to some reflective subcategory.  In recent years there has been an increased interest in localizations on the category of groups (see
\cite{akhtiamov2019right},
\cite{flores2018localizations}, 
\cite{libman2000cardinality}, \cite{casacuberta2000structures}, \cite{osullivan2012localizations}). 
In personal communication, Emmanuel Dror Farjoun formulated the following conjectures.

\begin{conjecture}
Any localization on the category of groups sends a nilpotent group to a nilpotent group.
\end{conjecture}

\begin{conjecture}
For any localization $L$ on the category of groups and for any finite $p$-group $G$ the map $G\to LG$ is surjective.
\end{conjecture}

\begin{conjecture}
For any localization $L$ on the category of groups and any group $G$ the cokernel of the map $G\to LG$ is abelian.
\end{conjecture}

Generally Conjectures 1 and 2 are open but there are some particular positive results in this direction. Conjecture 1 was proven for nilpotent groups of class $\leq 2$ by Libman \cite{libman2000cardinality}. Conjectures 1 and 2 were proved for right exact localizations by Akhtiamov, Pavutnitskiy and Ivanov  \cite{akhtiamov2019right}. This work is devoted to construction of a counterexample for Conjecture 3. 

Let us remind the definition of the Baumslag rationalization. A group $G$ is called uniquely divisible (or a  $\QQ$-group), if for any  $n\in \ZZ$ the map $G\to G, g\mapsto g^n$ is bijective. The Baumslag rationalization of a group $G$ is the universal map to a uniquely divisible group
\[
\eta_G: G\longrightarrow \bau(G). 
\]
Baumslag rationalization exists for any group and defines a localization $(\bau,\eta)$ on the category of groups.

A homomorphism $f:H\to G$ is called a \emph{Nikolov-Segal map} if it satisfies 
\[ [G,G] = [G,{\rm Im}(f)].\] 
Note that the cokernel of a Nikolov-Segal map is abelian. 
The term 'Nikolov-Segal map' is motivated by a theorem of Nikolov and Segal 
 \cite[Th. 1.7]{nikolov2012generators} which implies that for a finitely generated group $G$ the homomorphism to its profinite completion $G\to \widehat{G}$  is a Nikolov-Segal map. 
It is also known \cite[Prop. 8.1]{akhtiamov2019right} that if $G$ is nilpotent, then the map $\eta_G:G\to \bau(G)$ is a Nikolov-Segal map. 

For any group $G$ we consider the cokernel of the map to its Baumslag rationalization 
\[{\sf CBau}(G):={\sf Coker}(\eta_G :G\to \bau(G)).\]
In other words, ${\sf CBau}(G)$ is the quotient of $\bau(G)$ by the normal closure of ${\rm Im}(\eta_G).$  Since $\eta_G:G\to \bau(G)$ is a Nikolov-Segal map for a nilpotent group $G$, ${\sf CBau}(G)$ is abelian in this case. So Conjecture 3 holds for the case of Baumslag rationalization and a nilpotent group $G.$ The aim of this paper is to prove that Conjecture 3 fails for the Baumslag rationalization and the free group of rank 2. Moreover, we prove the following theorem.  

\begin{theorem*}\label{th_main}
Let $F_2$ be the free group of rank 2. Then  ${\sf CBau}(F_2)$ contains a free  subgroup of countable rank. In particular, Conjecture 3 fails.  
\end{theorem*}

Since ${\sf CBau}(F_2)$ contains a free  subgroup of countable rank, it is not abelian and not even solvable. Moreover, it is not in any proper variety of groups.

It seems that  right exact localizations on the category of groups that were studied in \cite{akhtiamov2019right} have much nicer properties than arbitrary ones. For example, as we noticed above, Conjectures 1 and 2 hold for right exact localizations.  
However, Conjecture 3 fails even in this case because  the Baumslag rationalization is a right exact localization \cite[Th 4.1]{akhtiamov2019right}. But we still have some reasons to believe that  Conjecture 3 holds for right exact localizations and nilpotent groups. Moreover, we believe that in this case the map $G\to LG$ is a Nikolov-Segal map. 
We leave it here in the form of new conjecture. 

\begin{conjecture}
For any right exact localization $L$ on the category of groups and any nilpotent group $G$ the map $G\to LG$ is a Nikolov-Segal map.
\end{conjecture}

\section*{Acknowledgments}
I am grateful to Emmanuel Farjoun, Danil Akhtiamov and Fedor Pavutnitskiy for helpful discussions.

\section*{Proof of the Theorem}

A group $G$ is called divisible (resp. uniquely divisible), if for any  $n\in \ZZ$ the map $G\to G, g\mapsto g^n$ is surjective (resp. bijective). A group is uniquely divisible if and only if it is local with respect to the homomorphism $\ZZ \hookrightarrow \QQ.$ The existence of the Baumslag localization for any group follows from \cite[Cor. 1.7]{casacuberta1992orthogonal}.

The {\it free uniquely divisible group} (or \emph{the free $\QQ$-group}) generated by a set $X$ is defined as $$F^\QQ(X)=\bau( F(X)),$$ where $F(X)$ is the free group generated by $X$. Recently  Jaikin-Zapirain proved a conjecture of Baumslag that  $F^\QQ(X)$ is residually torsion-free nilpotent \cite{jaikin2020free}. It is easy to see that any map from the set $X$ to a uniquely divisible group $\varphi : X \to G$ can be uniquely extended to a homomorphism $\Phi: F^\QQ(X)\to G.$ Gilbert Baumslag proved the following version of this statement for all divisible groups. 

\begin{theorem}[{Baumslag, \cite[Th. 39.6]{baumslag1960some}}] \label{th_baumslag} Any map from a set to a (possibly non-uniquely) divisible group $\varphi:X\to G$ can be extended to a (possibly non-unique) homomorphism $\Phi: F^\QQ(X)\to G.$
\end{theorem}

\begin{lemma}\label{lemma_subquotient} Let $\varphi:X\to G$ be a map from a set $X$ to a divisible group $G.$ Assume that all elements of $\varphi(X)$ are torsion elements. Then the subgroup $\langle \varphi(X) \rangle$ is a subquotient  of ${\sf CBau}(F(X)).$ 
\end{lemma}
\begin{proof}
By Theorem \ref{th_baumslag} we have a homomorphism $\Phi: F^\QQ(X)\to G$ such that $\Phi(x)=\varphi(x) $ for $x\in X.$ Since $\varphi(X)$ consists of torsion elements, for any $x\in X$ there exists a number $n_x\geq 1$  such that $\varphi(x)^{n_x}=1.$ By the universal property of $F^\QQ(X)$ there exists a unique homomorphism  $\Psi: F^\QQ(X)\to F^\QQ(X)$ such that $\Psi(x)=x^{n_x}$ for any $x\in X.$ Similarly there exists a unique homomorphism $ \Psi': F^\QQ(X)\to F^\QQ(x)$ such that $\Psi'(x)=x^{\frac{1}{n_x}}$ for any $x\in X.$ It is easy to see that $\Psi^{-1}=\Psi',$ and hence $\Psi$ is an automorphism. The composition 
$$\Phi \Psi : F^\QQ(X)\longrightarrow G $$
satisfies $\Phi\Psi(x)=\varphi(x)^{n_x} =1$ for all $x\in X.$ Thus the homomorphism  $\Phi\Psi\eta_{F(X)}$ is trivial, and hence, the homomorphism $\Phi\Psi$ induces a homomorphism  $$\Theta:{\sf CBau}(F(X))\to G.$$ 
Since $\Psi$ is an automorphism, we have 
$$ \langle \varphi(X) \rangle \subseteq {\rm Im}(\Phi)  = {\rm Im}( \Phi \Psi ) = {\rm Im}(\Theta).$$
The assertion follows. 
\end{proof}

In order to prove the theorem we need to add two additional well known ingredients: some information about subgroups of the special orthogonal group $SO(3),$ which is an interesting example of a divisible but not uniquely divisible group; and some information about the kernel of the map from the free product to the direct product for any groups  $G*H\to G\times H.$ 

Consider the special orthogonal group $SO(3),$ the group of rotations of $\mathbb{R}^3.$ Following Radin and Sadun \cite{radin19992},
for each positive integer $p$ we denote by $R_x^{2\pi/p}$ the rotation around the axis $x$ on the angle $2\pi/p,$ and for any positive integer $q$ we denote by $R_z^{2\pi/q}$ the rotation around the axis $z$ on on the angle $2\pi/q.$ We also denote by 
$G(p,q)$ the subgroup of $SO(3)$ generated by $R_x^{2\pi/p}$ and $R_z^{2\pi/q}$
$$G(p,q) = \langle  R_x^{2\pi/p},R_z^{2\pi/q} \rangle \subseteq SO(3).$$
\begin{theorem}[{Radin, Sadun, \cite[Th.2]{radin19992}}]\label{th_radin}
If $p$ and $q$ are odd, then  $$G(p,q)\cong \ZZ/p*\ZZ/q.$$ \end{theorem}

For any groups $G,H$ we set 
$$G \square H:={\rm Ker}(G*H\to G\times H).$$ The following statement is well-known and the standard reference is   \cite{levi1941commutatorgroup} (see also \cite{gilbert1989non}).

\begin{theorem}[{Levi  \cite{levi1941commutatorgroup}}]\label{th_levi}
For any groups $G,H$ the group $G\square H$ is a free group freely generated by the commutators $[g,h]$ for $g\in G\setminus\{1\}, h\in H\setminus \{1\}$ 
\[
G\square H \cong F((G\setminus\{1\})\times (H\setminus \{1\})).
\]
\end{theorem}

Now we are ready to prove our theorem.

\begin{proof}[Proof of Theorem] Note that the group $SO(3)$ is divisible. Indeed, any its element is a rotation around some line on some angle $\alpha$, and the  rotations around the same line on the angles $\alpha/n$ are roots of this element. Consider two odd positive integers  $p,q\geq 3$ and the map $\varphi:\{x,y\}\to SO(3)$ given by $\varphi(x)=R_x^{2\pi/p}$ and $\varphi(y)=R_y^{2\pi/q}.$ 
Then by Lemma  \ref{lemma_subquotient} we obtain that $G(p,q)$ is a subquotient of ${\sf CBau}(F_2),$ where $F_2=F(x,y)$ is the free group of rank 2. 
Using Theorem \ref{th_radin}, we obtain that  $\ZZ/p*\ZZ/q$ is a subquotient of ${\sf CBau}(F_2).$ Hence $\ZZ/p\square \ZZ/q$ is also a subquotient of ${\sf CBau}(F_2).$ 
By Theorem \ref{th_levi} the group $\ZZ/p\square \ZZ/q$ is a free group with $(p-1)(q-1)$ generators. 
It follows that a free group of rank $(p-1)(q-1)$  is a subquotient of ${\sf CBau}(F_2).$ 
Since any epimorphism to a free group splits, a free group of rank $(p-1)(q-1)$  is a subgroup of ${\sf CBau}(F_2).$ In particular, the free group of rank two $F_2$ can be embedded into ${\sf CBau}(F_2).$ It is well known that the commutator subgroup of $F_2$ is a free group of countable rank $[F_2,F_2]\cong F_\infty.$ Therefore $F_\infty$ can be embedded into ${\sf CBau}(F_2).$
\end{proof}

\printbibliography

\end{document}